\documentclass[12pt,a4paper,twoside]{article}
\usepackage{amsmath,amsthm,amssymb,amsfonts,amsbsy}

\usepackage{color, colortbl} % colouring text
\usepackage{float} % determining exact positions of figures and tables using [H]
\usepackage[margin=2.5cm]{geometry} % setting page margins
\usepackage{cite}
\usepackage{authblk}
\usepackage[utf8]{inputenc}
\usepackage{pgf,tikz}
\usetikzlibrary{arrows}

\newtheorem{theorem}{Theorem}[section]

\newtheorem{corollary}[theorem]{Corollary}

\numberwithin{equation}{section}
\usepackage[utf8]{inputenc}
\usepackage{pstricks-add}

\newcommand{\beq}{\begin{eqnarray}}
\newcommand{\eeq}{\end{eqnarray}}

\newcommand{\beqs}{\begin{eqnarray*}}
\newcommand{\eeqs}{\end{eqnarray*}}

\allowdisplaybreaks

\usepackage{mathtools}

\title{\bf  Maximum Zagreb Indices Among All $p-$Quasi $k-$Cyclic Graphs}

\author{ \bf Ali Ghalavand\thanks{Corresponding author (alighalavand@grad.kashanu.ac.ir)}~ and  Ali Reza Ashrafi
}

\affil{ \normalsize
    { \it Department of Pure Mathematics, Faculty of Mathematical Sciences, University of Kashan, Kashan 87317--53153, I. R. Iran}}

\begin{document}

\maketitle

\begin{abstract}
A simple connected graph $G$ is called a $p-$quasi $k-$cyclic graph, if there exists a subset $S$ of vertices such that $|S| = p$, $G \setminus S$ is $k-$cyclic and there is no a subset $S^\prime$ of $V(G)$ such that $|S^\prime| < |S|$ and $G \setminus S^\prime$ is $k-$cyclic. The aim of this paper is to characterize graphs with maximum values of Zagreb indices among all $p-$quasi $k-$cyclic graphs with $k \leq 3$.

\vskip 3mm


\vskip 3mm


\end{abstract}

\bigskip

\section{Introduction}

Throughout this paper all graphs are assumed to be simple and connected. We use the notations $P_n$, $C_n$, $S_n$ and $K_n$ to denote the $n-$vertex path, cycle, star and complete graph, respectively. The cyclomatic number of $G$ is defined as $C(G) = |E(G)| - |V(G)| + 1$ and if $C(G) = k$ then we say that $G$ is $k-$cyclic. In the special cases that $k = 0, 1, 2, 3$ the graph is called tree, unicyclic, bicyclic and tricyclic, respectively. The set of all $k-$cyclic graphs on a fixed vertex set of size $n$ is denoted by $C^k(n)$ and $N_G(v)$ is the set of all neighbours of $v$ in $G$.

A vertex of degree 1 is named a pendant vertex and an edge containing a pendant vertex is called a pendant edge of $G$. The maximum and minimum degrees of vertices are denoted by $\Delta = \Delta(G)$ and $\delta = \delta(G)$, respectively. A simple connected graph $G$ is called a $p-$quasi $k-$cyclic graph, if there exists a subset $S$ of vertices such that $|S| = p$, $G \setminus S$ is $k-$cyclic and there is no a subset $S^\prime$ of $V(G)$ such that $|S^\prime| < |S|$ and $G \setminus S^\prime$ is $k-$cyclic. The set of all such graphs is denoted by $Q_pC^k(n)$. Our other notations are standard and can be taken from the standard books on graph theory.

The Zagreb indices are the most  degree-based graph invariants were introduced by Gutman and Trinajsti\'c \cite{g0}. These graph invariants can be defined as:
\begin{eqnarray*}
M_{1}(G) &=& \sum_{v\in V(G)}(d(v))^{2},\\
M_{2}(G) &=& \sum_{uv\in E(G)}(d(u)d(v)).
\end{eqnarray*}

The following theorem is useful in our main results:

\begin{theorem}\label{RA1}
$($See \cite{g1,g2}$)$. Let $T$ be a tree of order $n$. If $T$ is different from
$S_n$, then $M_1(T ) < M_1(S_n)$ and $M_2(T) < M_2(S_n)$.
\end{theorem}

Let $U^3_n$ be the unicyclic graph obtained from the cycle $C_3$  by
attaching $n-3$ pendent edges to the same vertex on $C_3$.

\begin{theorem}\label{RA2}
$($ See \cite{z1,y1}$)$. $U^3_n$ is the unique graph with the largest Zagreb
indices $M_1$ and $M_2$ among all unicyclic graphs with $n$ vertices.
\end{theorem}
%%%%%%%%%%%%%%%%%
\begin{figure}
\begin{center}
\begin{tikzpicture}[line cap=round,line join=round,>=triangle 45,x=1.0cm,y=1.0cm]
\clip(-10,0) rectangle (-4,4);
\draw(-6.53,1.57) circle (0.57cm);
\draw (-6.48,2.14)-- (-6.54,1);
\draw (-6.48,2.14)-- (-7.01,2.64);
\draw (-6.48,2.14)-- (-6.74,2.67);
\draw (-6.48,2.14)-- (-5.96,2.64);
\draw (-6.85,3.26) node[anchor=north west] {$n-4$};
\draw (-7.,0.73) node[anchor=north west] {$B_n^{3,3}$};
\begin{scriptsize}
\fill [color=black] (-6.54,1) circle (1.5pt);
\fill [color=black] (-6.48,2.14) circle (1.5pt);
\fill [color=black] (-7.09,1.54) circle (1.5pt);
\fill [color=black] (-5.96,1.53) circle (1.5pt);
\fill [color=black] (-7.01,2.64) circle (1.5pt);
\fill [color=black] (-6.74,2.67) circle (1.5pt);
\fill [color=black] (-6.53,2.7) circle (1.5pt);
\fill [color=black] (-6.28,2.7) circle (1.5pt);
\fill [color=black] (-6.11,2.69) circle (1.5pt);
\fill [color=black] (-5.96,2.64) circle (1.5pt);
\end{scriptsize}
\end{tikzpicture}
\caption{ The bicyclic graph $B_n^{3,3}$ in Theorem \ref{RA3}.\label{sh1}}
\end{center}
\end{figure}
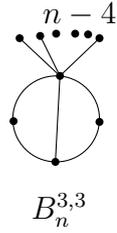

\vskip 3mm

%%%%%%%%%%%%%%%%%%
\begin{theorem}\label{RA3}
$($ See \cite{d1}$)$. $B_n^{3,3}$  is the unique graph with  the largest Zagreb indices $M_1$ and $M_2$ among all bicyclic graphs with $n$ vertices, see Figure \ref{sh1}.
\end{theorem}

%%%%%%%%%%%%%%%%%%%%%%
\begin{figure}
\begin{center}
\begin{tikzpicture}[line cap=round,line join=round,>=triangle 45,x=1.0cm,y=1.0cm]
\clip(-9,1) rectangle (8,6);
\draw (-4,4)-- (-5,3);
\draw (-4,4)-- (-3,3);
\draw (-4.02,3.37)-- (-3,3);
\draw (-4.02,3.37)-- (-5,3);
\draw (-5,3)-- (-3,3);
\draw (-4,2)-- (-5,3);
\draw (-4,2)-- (-3,3);
\draw (-5,3)-- (-5.73,3.6);
\draw (-5,3)-- (-5.85,3.33);
\draw (-5,3)-- (-5.89,2.59);
\draw (-7.03,3.33) node[anchor=north west] {$n-5$};
\draw (-5.94,1.56) node[anchor=north west] {$q_n(n-4,1,1,1,1)$};
\draw (0,4)-- (-1,2);
\draw (0,4)-- (1,2);
\draw (1,2)-- (-1,2);
\draw (0,2.8)-- (1,2);
\draw (0,2.8)-- (0,4);
\draw (0,2.8)-- (-1,2);
\draw (0,4)-- (-0.65,4.66);
\draw (0,4)-- (-0.39,4.69);
\draw (0,4)-- (0.48,4.68);
\draw (-0.64,5.36) node[anchor=north west] {$n-4$};
\draw (-1.28,1.49) node[anchor=north west] {$K_n(n-3,1,1,1)$};
\begin{scriptsize}
\fill [color=black] (-4,4) circle (1.5pt);
\fill [color=black] (-5,3) circle (1.5pt);
\fill [color=black] (-3,3) circle (1.5pt);
\fill [color=black] (-4.02,3.37) circle (1.5pt);
\fill [color=black] (-4,2) circle (1.5pt);
\fill [color=black] (-5.73,3.6) circle (1.5pt);
\fill [color=black] (-5.85,3.33) circle (1.5pt);
\fill [color=black] (-5.85,3.15) circle (1.5pt);
\fill [color=black] (-5.86,2.94) circle (1.5pt);
\fill [color=black] (-5.87,2.81) circle (1.5pt);
\fill [color=black] (-5.89,2.59) circle (1.5pt);
\fill [color=black] (0,4) circle (1.5pt);
\fill [color=black] (-1,2) circle (1.5pt);
\fill [color=black] (1,2) circle (1.5pt);
\fill [color=black] (0,2.8) circle (1.5pt);
\fill [color=black] (-0.65,4.66) circle (1.5pt);
\fill [color=black] (-0.39,4.69) circle (1.5pt);
\fill [color=black] (-0.13,4.71) circle (1.5pt);
\fill [color=black] (0.08,4.71) circle (1.5pt);
\fill [color=black] (0.26,4.7) circle (1.5pt);
\fill [color=black] (0.48,4.68) circle (1.5pt);
\end{scriptsize}
\end{tikzpicture}
\caption{ The tri-cyclic graphs  $q_n(n-4,1,1,1,1)$ and $K_n(n-3,1,1,1)$ in Theorem \ref{RA4}.\label{sh2}}
\end{center}
\end{figure}
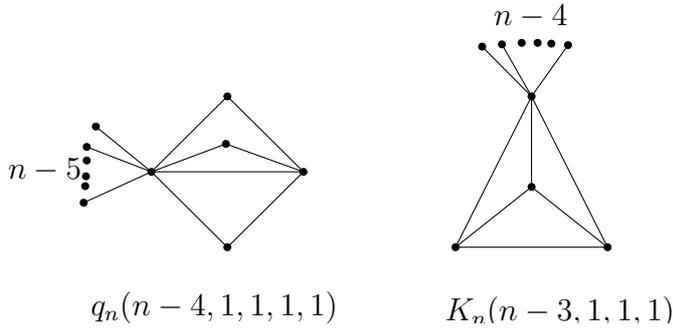
%%%%%%%%%%%%%%%%%%%%%%%%%%

\begin{theorem}\label{RA4}
$($ See \cite{a1}$)$. Among all tri-cyclic graphs with $n(\geq 5)$ vertices,
\begin{enumerate}
\item $K_n(n-3,1,1,1)$ and $q_n(n-4,1,1,1,1)$ have the maximum values of first Zagreb index.

\item The graph $K_n(n-3,1,1,1)$ has maximum value of the second Zagreb index.
\end{enumerate}
\end{theorem}

Suppose $G$ and $H$ are two simple graphs. The union $G \cup H$ is a graph with vertex set $V(G) \cup V(H)$ and edge set $E(G) \cup E(H)$. The join of $G$ and $H$ is a graph with the same vertex set as $G \cup H$ and $E(G + H)$ = $E(G)$ $\cup$ $E(H)$ $\cup$ $\{ xy \mid x \in V(G), y \in V(H) \}$. Our other notations are standard and can be taken from \cite{im}.

\section{  Main Results }

The aim of this paper is to characterize graphs with maximum values of Zagreb indices among all $p-$quasi $k-$cyclic graphs with $k \leq 3$.

\begin{theorem}\label{A1}

Let $G$ be a  $p-$quasi $k-$cyclic graph. If $S\subset V(G)$, $|S|=p$ and $G-S \in C^k(n-p)$, then

\begin{enumerate}
\item $M_1(G)$ $\leq$ $M_1(G-S)$ $+$ $p(4k+n^2+2n+p(n-4)-p^2-3)$,

\item $M_2(G)$ $\leq$ $M_2(G-S)$ $+$ $pM_1(G-S)$ $+$ $(k+n-p-1)(p^2+2p(n-1))$ $+$ $\frac{p(p-1)(n-1)}{2}+p^2(n-p)(n-1)$,
\end{enumerate}
with equality in each if and only if $G\cong (G-S) + K_p$.
\end{theorem}

\begin{proof}
To prove $(1)$, we assume that $u\in V(G-S)$ and define $l_u$ to be the number of vertices in $S$ adjacent to $u$. By definition of $M_1$, $$M_1(G) = \sum_{u\in V(G-S)}d_G^2(u)+\sum_{u\in S}d_G^2(u) = \sum_{u\in V(G-S)}(d_{G-S}(u)+l_u)^2+\sum_{u\in S}d_G^2(u).$$ By simplifying this equalities,

\begin{eqnarray*}
M_1(G)&=&\sum_{u\in V(G-S)}(d_{G-S}^2(u)+l^2_u+2d_{G-S}(u)l_u)+\sum_{u\in S}d_G^2(u)\\
&=&M_1(G-S)+\sum_{u\in V(G-S)}l^2_u+\sum_{u\in V(G-S)}2d_{G-S}(u)l_u+\sum_{u\in S}d_G^2(u)\\
&\leq &M_1(G-S)+\sum_{u\in V(G-S)}p^2+\sum_{u\in V(G-S)}2d_{G-S}(u)p+\sum_{u\in S}(n-1)^2\\
&=&M_1(G-S)+(n-p)p^2+4p(k+n-p-1)+p(n-1)^2\\
&=&M_1(G-S)+p(4k+n^2+2n+p(n-4)-p^2-3).
\end{eqnarray*}

The equality holds if and only if for each $u\in V(G-S)$, $l_u=p$ and for every vertex $u\in S$, we have $d_G(u)=n-1$. This condition is satisfied if and only if $G\cong (G-S)+ K_p$, proving the first part of theorem.

To prove $(2)$, we assume that $u^*v^*$ is an edge of $G$ such that $u^*\in V(G-S)$ and $v^*\in S$. By definition of $M_2$,

\begin{eqnarray*}
M_2(G)&=&\sum_{uv\in E(G-S)}d_G(u)d_G(v)+\sum_{uv\in E(G-(G-S))}d_G(u)d_G(v)\\
&+& \sum_{u^*v^*\in E(G)}d_G(u^*)d_G(v^*)\\
&=&\sum_{uv\in E(G-S)}(d_{G-S}(u)+l_u)(d_{G-S}(v)+l_v)+\sum_{uv\in E(G-(G-S))}d_G(u)d_G(v)\\
&+&\sum_{u^*v^*\in E(G)}(d_{G-S}(u^*)+l_{u^*})d_G(v^*).
\end{eqnarray*}

By simplifying the last equality,

\begin{eqnarray*}
M_2(G) &=& \sum_{uv\in E(G-S)}(d_{G-S}(u)d_{G-S}(v)+d_{G-S}(u)l_v+d_{G-S}(v)l_u+l_ul_v)\\
&+&\sum_{uv\in E(G-(G-S))}d_G(u)d_G(v)+\sum_{u^*v^*\in E(G)}(d_{G-S}(u^*)d_{G}(v^*)+d_{G}(v^*)l_{u^*}).
\end{eqnarray*}

Therefore,

\begin{eqnarray*}
M_2(G)&=&M_2(G-S)+\sum_{uv\in E(G-S)}(d_{G-S}(u)l_v+d_{G-S}(v)l_u+l_ul_v)\\
&+&\sum_{uv\in E(G-(G-S))}d_G(u)d_G(v)+\sum_{u^*v^*\in E(G)}(d_{G-S}(u^*)d_{G}(v^*)+d_{G}(v^*)l_{u^*})\\
&\leq&M_2(G-S)+p\sum_{uv\in E(G-S)}(d_{G-S}(u)+d_{G-S}(v)+p)\\
&+&\sum_{uv\in E(G-(G-S))}(n-1)^2+\sum_{u^*v^*\in E(G)}(d_{G-S}(u^*)(n-1)+(n-1)p)\\
&=&M_2(G-S)+pM_1(G-S)+(k+n-p-1)p^2+\frac{p(p-1)(n-1)^2}{2}\\
&+&2p(n-1)(k+n-p-1)+p^2(n-p)(n-1)\\
&=&M_2(G-S)+pM_1(G-S)+(k+n-p-1)(p^2+2p(n-1))\\
&+&\frac{p(p-1)(n-1)}{2}+p^2(n-p)(n-1).
\end{eqnarray*}

The equality is satisfied if and only if for each $u\in V(G-S)$, $l_u=p$ and for every vertex $u\in S$, $d_G(u)=n-1$. This condition is also equivalent to the fact that $G\cong (G-S)+ K_p$. This completes the proof.
\end{proof}

\begin{theorem}\label{A2}
Suppose $A=\{H_1,H_2,...,H_r\}\subset C^k(n-p)$, $H\in C^k(n-p)\setminus A$, $B=\{H_i+ K_p|i=1,2,...,r\}$ and
$G\in Q_pC^k(n)\setminus B$. If $M_1(H)<M_1(H_1)=...=M_1(H_r)$ and $M_2(H)<M_2(H_1)=...=M_2(H_r)$, then
\begin{enumerate}
\item $M_1(G)$ $<$ $M_1(H_1+ K_p)$ $=$ $\cdots$ = $M_1(H_r+ K_p)$,

\item $M_2(G)$ $<$ $M_2(H_1+ K_p)$ = $\cdots$ = $M_2(H_r+ K_p)$.
\end{enumerate}
\end{theorem}

\begin{proof}
By Theorem \ref{A1}(1), for each $i$, $1 \leq i \leq r$, $M_1(H_i+ K_p)$ = $M_1(H_i)$ + $p(4k+n^2+2n+p(n-4)-p^2-3).$ Since
$G\not\in B$, for every subset $S$ of $V(G)$ with this property that $G-S\in C^k(n-p)$, we have $G-S\not\in A$ or $G-S\in A$ and
$G\neq (G-S)+ K_p$. Thus, by Theorem \ref{A1}(1), $M_1(G)<M_1(H_1+ K_p)=...=M_1(H_r+ K_p)$. To prove the second part, we apply Theorem \ref{A1}(2) and a similar argument as above.
\end{proof}

From Theorems \ref{RA1}, \ref{RA2}, \ref{RA3}, \ref{RA4} and \ref{A2}, we have the following corollary.

\begin{corollary}
Suppose $n$ is a given positive integer and $G \in Q_pC^k(n)$. Then,
\begin{enumerate}
\item  If $k = 0$ and $n \geq p + 2 $ then $M_1(G) \leq M_1(S_{n-p}+ K_p)$ and $M_2(G) \leq M_2(S_{n-p}+ K_p)$. Hence $S_{n-p}+ K_p$ has the maximum first and second Zagreb indices in the class $Q_pC^0(n)$ with $n \geq p + 2 $.

\item  If $k = 1$ and $n \geq p + 3 $ then $M_1(G) \leq M_1(U_{n-p}^3+ K_p)$ and $M_2(G) \leq M_2(U_{n-p}^3+ K_p)$. Hence $U_{n-p}^3+ K_p$ has the maximum first and second Zagreb indices in the class $Q_pC^1(n)$ with $n \geq p + 3 $.

\item  If $k = 2$ and $n \geq p + 4 $ then $M_1(G) \leq M_1(B_{n-p}^{3,3}+ K_p)$ and $M_2(G) \leq M_2(B_{n-p}^{3,3}+ K_p)$. Hence $B_{n-p}^{3,3}+ K_p$ has the maximum first and second Zagreb indices in the class $Q_pC^2(n)$ with $n \geq p + 4 $.

\item  If $k = 3$ and $n \geq p + 5 $ then $M_1(G) \leq M_1(K_{n-p}(n-p-3,1,1,1)+ K_p)$ = $M_1(q_{n-p}(n-p-4,1,1,1,1)+ K_p)$. Hence $K_{n-p}(n-p-3,1,1,1)+ K_p$ and $q_{n-p}(n-p-4,1,1,1,1)+ K_p$ have the maximum first Zagreb index in the class $Q_pC^3(n)$ with $n \geq p + 5 $.

\item  If $k = 3$ and $n \geq p + 5 $ then $M_2(G) \leq M_2(K_{n-p}(n-p-3,1,1,1)+ K_p)$. Hence $K_{n-p}(n-p-3,1,1,1)+ K_p$ has the maximum second Zagreb index in the class $Q_pC^3(n)$ with $n \geq p + 5 $.
\end{enumerate}
\end{corollary}

\vskip 3mm

\noindent\textbf{Acknowledgement.} This research is partially supported by the University of Kashan.

\end{document}